\documentclass[10pt]{amsart}

\usepackage{times,amsmath,amsbsy,amssymb,amscd,mathrsfs}
\usepackage{slashbox}\usepackage{graphicx,subfigure,epstopdf,wrapfig,chemarrow}
\usepackage{algorithm2e} 
\usepackage{multicol,multirow}
\usepackage{mathtools}
\usepackage[usenames,dvipsnames,svgnames,table]{xcolor}
\definecolor{myBlue}{rgb}{0.0,0.0,0.55}
\definecolor{green}{rgb}{0.0,0.7,0.2}
\usepackage[pdftex,colorlinks=true,citecolor=myBlue,linkcolor=myBlue]{hyperref}

\usepackage{comment,enumerate,multicol,xspace}

\newtheorem{theorem}{Theorem}[section]
\newtheorem{lemma}[theorem]{Lemma}

\newtheorem{remark}[theorem]{Remark}

\newcommand{\bs}{\boldsymbol}
\newcommand{\mcal}{\mathcal}

\newcommand{\vertiii}[1]{{\left\vert\kern-0.25ex\left\vert\kern-0.25ex\left\vert #1 
    \right\vert\kern-0.25ex\right\vert\kern-0.25ex\right\vert}}

\begin{document}
\title[Convergence analysis for saddle point systems]{Convergence Analysis for A Class of Iterative Methods for Solving Saddle Point Systems}
\thanks{L. Chen was supported by NSF Grant DMS-1418934, in part by the Sea Poly Project of Beijing Overseas Talents and in part by National Natural Science Foundation of China (Grant No. 11671159). Y. Wu was supported by the National Natural Science Foundation of China (11501088), the Fundamental Research Funds for the Central Universities of China (ZYGX2015J097, UESTC)}
\author{Long Chen and Yongke Wu}

\begin{abstract}
Convergence analysis of a nested iterative scheme proposed by Bank,Welfert and Yserentant (BWY) ([Numer. Math., 666: 645-666, 1990])  for solving saddle point systems is presented. It is shown that this scheme converges under weaker conditions: the contraction rate for solving the $(1,1)$ block matrix is bound by $(\sqrt{5}-1)/2$. Similar convergence result is also obtained for a class of inexact Uzawa method with even weaker contraction bound $\sqrt{2}/2$. Preconditioned generalized minimal residual (GMRes) method using the BWY iteration as a preconditioner is shown to converge with realistic assumptions.
\end{abstract}

\date{\today}


\address[L.~Chen]{Department of Mathematics, University of California at Irvine, Irvine, CA 92697, USA\\
Beijing Institute for Scientific and Engineering Computing, Beijing University of Technology, Beijing, 100124, China}
\email{chenlong@math.uci.edu}

\address[Y.~Wu]{School of Mathematical Sciences, University of Electronic Science and Technology of China, Chengdu 611731, China.}
\email{wuyongke1982@sina.com}

\maketitle

\section{Introduction}
Saddle point systems are often arising in a variety of scientific and engineering applications such as mixed finite element methods for elliptic equations (e.g. Poisson, Stokes, and elasticity problems) and constrained optimization problems where a Lagrange multiplier is used to impose the constraint~\cite{Benzi.M;Golub.G;Liesen.J2005,Boffi2013}. Solving saddle point systems is thus an important topic in the scientific computing. 

We shall consider iterative methods for solving the following saddle point system
\begin{equation}
\label{saddle_pro}
\begin{pmatrix}
A & B^{\intercal} \\ 
B  & - C
\end{pmatrix}
\begin{pmatrix}
u \\ p
\end{pmatrix} = 
\begin{pmatrix}
f \\ g
\end{pmatrix},
\end{equation}
where $A \in \mathbb R^{n \times n}$ is a symmetric positive definite (SPD) matrix, $B \in \mathbb R^{m\times n}\ (m\leq n)$ is a full rank matrix, and $ C \in \mathbb R^{m \times m}$ is symmetric positive and semi-definite matrix. We assume that system~\eqref{saddle_pro} is well-posed and thus $BA^{-1}B^{\intercal} + C$, the Schur complement of $A$, is also symmetric and positive definite. 

A class of iterative methods was introduced by Bank, Welfert, and Yserentant~\cite{Bank1990}. Given the current approximation $(u^k, p^k)$, to compute $(u^{k+1}, p^{k+1})$, the BWY method consists of three steps:
\begin{align}
\label{BWY_1} u^{k+1/2} &= u^{k} + R_A(f - Au^k - B^{\intercal}p^k),\\
\label{BWY_2} p^{k+1} &= p^{k} - R_S(g - Bu^{k+1/2} + Cp^k),\\
\label{BWY_3} u^{k+1} &= u^{k} + R_A(f - Au^k - B^{\intercal}p^{k+1}),
\end{align}
where $R_A \in \mathbb R^{n\times n}$ is an SPD approximation of $A^{-1}$ but the SPD matrix $R_S^{-1} \in \mathbb R^{m \times m}$ is for a different Schur complement $S: = B\, R_AB^{\intercal}+C$.

Let $\delta = \rho(I - R_AA)$ and $\gamma = \rho(I - R_SS)$, where $\rho(\cdot)$ is the spectral radius of a matrix. In \cite{Bank1990}, Bank, Welfert, and Yserentant proved that under the norm
$$
\vertiii{\cdot}^2 := \|\cdot\|_{R_A^{-1}}^2 +\frac{1-\gamma}{2(1+\gamma)} \|\cdot\|_{S}^2, 
$$ 
the error operator $\mathcal E$ of the BWY method satisfies
\begin{equation}\label{BWYestimate}
\vertiii{\mathcal E} \leq \max\left\{\delta,\frac{2\gamma}{1 - \gamma}  \right\}.
\end{equation}
Therefore the BWY method is convergent when $\gamma < \frac{1}{3}$ and $\delta  < 1$. 
In \cite{Tong1998}, Tong and Sameh weakened the conditions stated in \cite{Bank1990} and showed that using another weight in the norm 
$$
\vertiii{\cdot}^2 := \|\cdot\|_{R_A^{-1}}^2 +  \frac{\gamma}{\delta(1+\gamma)} \|\cdot\|_{S}^2,
$$
the error operator $\mathcal E$ of the BWY method was bounded by
\begin{equation}\label{TB}
\left\{\begin{array}{ll}
\vertiii{\mathcal E} < 1 & \text{when}\quad \delta < 1,\gamma < \frac{1}{1+2\delta},\\
\vertiii{\mathcal E} < \delta & \text{when}\quad \delta < 1,\gamma \leq \frac{\delta}{2+\delta}.
\end{array}
\right.
\end{equation}
This improved convergence analysis explains why in practice the BWY method is often convergent even though $\gamma > \frac{1}{3}$. To achieve a robust convergence rate, however, restriction on $\gamma \leq \delta /(2+\delta)$ is still needed. It is also worthing to note that in \cite{Bank1990,Tong1998} only the case $C=0$ is considered. 

We shall present a different convergence analysis of the BWY method. Introduce $\bar R_A = 2R_A - R_AA\, R_A$, the so-called symmetrization of $R_A$, and the corresponding Schur complement $\bar S := B\bar R_AB^{\intercal} + C$. A crucial point is to view $R_S^{-1}$ as an approximation of $\bar S$ instead of $S$. 

Assume that 
$$R_A^{-1} > A, R_S^{-1} \geq \bar S, \text{ and }\delta < \frac{\sqrt{5}-1}{2},$$ 
then we shall prove that 
 \begin{align}\label{error}
 \|u - u^k\|_{R_A^{-1}}^2 + \|p-p^k\|_{R_S^{-1}}^2 \leq 9 \rho_1^{2k}\left( \|u-u^0\|_{R_A^{-1}}^{2} + \|p-p^0\|_{R_S^{-1}}^2 \right),
 \end{align} 
where $\rho_1 = \rho_1(\delta,\bar \gamma)<1$ with $\bar \gamma = \rho(I - R_S\bar S)$ (detailed formulae can be found in Section 2.1).
A convergence result similar to \eqref{error} can be also obtained with assumptions $R_A^{-1} >  A, R_S^{-1} \geq S, \text{ and }\delta < \frac{1}{2}.$
Comparing with the contraction result \eqref{BWYestimate} and \eqref{TB}, we relax the condition for $\bar \gamma$, which is simply $\bar \gamma <1$ as implied by the assumption $R_S^{-1} \geq \bar{S}$, but impose a stronger condition on the contraction factor $\delta$ for solving $A$. As $A$ is available, the assumption $R_A^{-1} > A$ can be easily satisfied by using multiplicative methods e.g. symmetric Gauss-Seidel iteration or V-cycle multigrid methods with symmetric Gauss-Seidel smoothers. The condition $R_S^{-1} \geq \bar S$ can be satisfied by re-scaling. 


We then consider a variant of the BWY method and call it symmetrized inexact Uzawa (SIUM) method:
\begin{align}
\label{SIU_1} u^{k+1/2} &= u^{k} + R_A(f - Au^k - B^{\intercal}p^k),\\
\label{SIU_2} p^{k+1} &= p^{k} - R_S(g - Bu^{k+1/2} + Cp^k),\\
\label{SIU_3} u^{k+1} &= u^{k+1/2} + R_A(f - Au^{k+1/2} - B^{\intercal}p^{k+1}).
\end{align}
The change is in the third step. In \eqref{SIU_3}, $u^{k+1/2}$ is used while in \eqref{BWY_3}, $u^k$ is used. The approximation $u^{k+1}$ obtained in \eqref{SIU_3} is thus expected to be better. Indeed we can prove a slightly better convergence result. Assume $R_A^{-1} >  A$, $R_S^{-1}\geq \bar S $, and $\delta < \frac{\sqrt{2}}{2}$, then
 \begin{align}\label{errorSIU}
 \|u - u^k\|_{R_A^{-1}}^2 + \|p-p^k\|_{R_S^{-1}}^2 \leq 9 \rho_2^{2k}\left( \|u-u^0\|_{R_A^{-1}}^{2} + \|p-p^0\|_{R_S^{-1}}^2 \right),
 \end{align} 
 where $\rho_2 = \rho_2(\delta,\bar\gamma)<1$. The improvement is the relaxed upper bound of $\delta$ from $(\sqrt{5}-1)/2$ to $\sqrt{2}/2$. 
Again a convergence result similar to \eqref{errorSIU} can be  obtained with assumptions $R_A^{-1} > A$, $R_S^{-1}\geq S$, and $\delta < \frac{1}{2}$.

We emphasize that the condition on $\delta$ bounded by $(\sqrt{5}-1)/2\approx 0.618$ or $\sqrt{2}/2\approx 0.707$ can be easily satisfied by, for example, using only one multigrid V-cycle for solving $A$. For the iterative solver of the Schur complement, we simply require that it is convergent, i.e., $\bar \gamma < 1$ (or $\gamma < 1$ with s slightly smaller upper bound for $\delta$). While to apply convergence results in \cite{Bank1990}, iterative solvers for the Schur complement equation should be convergent with a contraction rate less than $1/3$. In most scenario, comparing with solving $Ax=b$, the Schur complement equation is much harder to solve since the Schur complement may not be formed explicitly. 

We shall also apply our convergence analysis to a special class of inexact Uzawa method (IUM). Given the current approximation $(u^k, p^k)$, to compute $(u^{k+1}, p^{k+1})$, IUM consists of two steps
\begin{align}
\label{IUM_1} u^{k+1} &= u^{k} + \bar R_A(f - Au^k - B^{\intercal}p^k),\\
\label{IUM_2} p^{k+1} &= p^{k} - R_S(g - Bu^{k+1} + Cp^k),
\end{align}
Here we require $\bar R_A$, the inexact solver in IUM, is a symmetrized smoothers which can be thought of as applying a symmetric solver $R_A$ twice. Write two consecutive steps of IUM and rearrange the approximation of $u$ and $p$, we will get the symmetrized inexact Uzawa iteration; details can be found in Section 2.3. Therefore we can prove the contraction of a special class of IUM in a similar form of \eqref{errorSIU} with a relaxed assumption $\delta < \sqrt{2}/2$ comparing with the bound $\delta < 1/3$ in the literature~\cite{Bramble1997,Cheng.X2000,Hu.Q;Zou.J2002,Cao2003,Bacuta2006a}. 

When $R_A = A^{-1}$, both BWY and IUM methods can be interpreted as iterative methods for solving the Schur complement equation $(BA^{-1}B^{\intercal} + C) p = BA^{-1}f-g$. Step \eqref{IUM_2} updating pressure is often called the outer iteration and step \eqref{IUM_1} is the inner iteration. When $R_A$ is sufficiently close to $A^{-1}$, e.g. applying sufficiently many inner iterations, one could expect the convergence of these methods. Our results imply that in practice one inner iteration is usually enough.
 
These iterative methods can be used as preconditioners in the generalized minimal residual (GMRes) method for solving the saddle point system~\eqref{saddle_pro}. For example, the BWY method corresponds to the approximate block factorization (ABF) preconditioner. Using the knowledge we learned from the convergence analysis of the BWY method, we shall prove that the preconditioned GMRes with the ABF preconditioner is convergent under assumptions
\begin{equation}\label{eq:GMRESassumption}
\delta < 1, \quad \kappa_1 R_S^{-1} \leq S \leq \kappa_2 R_S^{-1}.
\end{equation}
We achieve this by showing the field-of-values-equivalences of matrices \cite{Elman1982,Saad2003,Loghin;Wathen2004}, which is a general approach to analyze Krylov subspace methods. 

The number of iterations of preconditioned GMRes will of course depend on $\delta$ and the spectral equivalent constants $\kappa_1$ and $\kappa_2$. To be an effective preconditioner, it is preferable these parameters are bounded uniformly to the size of the matrix. A uniform bound for $\delta$ can be easily obtained using a V-cycle multigrid method for $A$. In most applications, the difficulty is to construct $R_S$ which is spectrally equivalent to the Schur complement $S$ but easier to compute. Recent application to mixed finite element methods for elliptic systems can be found in~\cite{Rusten1996Interior,Mardal2011a,Chen;Wu;Zhong:Zhou2016,chen2016fast}. 

We conclude the introduction by the following notation. Recall that $R_A$ and $R_S$ are SPD matrices. The symmetrization $\bar R_A = 2R_A - R_AA\, R_A$ satisfies the relation 
$$
I - \bar R_AA = (I - R_AA)^2.
$$
Denote by
\begin{align*}
&\mathcal A & = &
\begin{pmatrix}
A & B^{\intercal} \\ 
B  & - C
\end{pmatrix} ,&
&\mathcal L & = & 
\begin{pmatrix}
I  &  0 \\ B\, R_A & I
\end{pmatrix} ,&
&\mathcal U  & = &
\begin{pmatrix}
I & R_A B^{\intercal} \\ 0 & I
\end{pmatrix} ,&
\\
& \mathcal H & = &
\begin{pmatrix}
R_A^{-1} & 0 \\ 0 & -R_S^{-1}
\end{pmatrix} ,&
&\mathcal D  & = &
\begin{pmatrix}
R_A^{-1}  & 0 \\ 0  & R_S^{-1}
\end{pmatrix} ,&
&\mathcal J & = & 
\begin{pmatrix}
I  &  0 \\ 0  & -I
\end{pmatrix}. &
\end{align*} 
Obviously $\mathcal L = \mathcal U^{\intercal}$. It is also straightforward to get the block factorization
\begin{equation}\label{eq:Ahat}
\hat{\mathcal A} : = \mathcal L\, \mathcal H \, \mathcal U = 
\begin{pmatrix}
R_A^{-1} & B^{\intercal} \\ 
B  & B\, R_A \, B^{\intercal} -R_S^{-1}
\end{pmatrix}.
\end{equation}
Using this block factorization, $\hat{\mathcal A}^{-1}$ can be efficiently computed by inverting a lower triangular system $\mathcal L\, \mathcal H$ and then a upper one $\mathcal U$. The sign of the diagonal block in matrix $\mathcal H$ is from the saddle point structure. $\mathcal D = \mathcal H \, \mathcal J$ is block diagonal and SPD, and usually used as the norm in the analysis of iterative methods and preconditioners for $\mathcal A$. Our analysis shows that $  \mathcal L\, \mathcal D \, \mathcal U $ defines a better norm, where $\mathcal L$ and $\mathcal U$ serve as a change of basis. 

Finally we remark that Sch\"oberl and Zulehner~\cite{Schoberl2003} and John, R\"ude, Wohlmuth and Zulehner~\cite{blocksmootherZulehner2016} proved that the BWY iteration is also a good smoother under the assumptions $R_A^{-1} \geq A$ and $R_S^{-1} \geq S$ and thus can be used to construct efficient multigrid methods for solving the saddle point system.

We use $(\cdot,\cdot)$ for the standard $l_2$-inner product of vectors. For any symmetric and positive definite matrix $M \in \mathbb R^{r\times r}$ and vectors $\bs x,\ \bs y \in \mathbb R^r$, we define 
$$
(\bs x,\bs y)_{M} = (M\bs x,\bs y)\qquad\text{and}\qquad \|\bs x\|_M = (\bs x,\bs x)_M^{1/2}.
$$ 
We say a matrix $T \in \mathbb R^{r\times r}$ is symmetric with respect to the inner product $(\cdot,\cdot)_M$, if 
$$
(T\bs x,\bs y)_M = (\bs x, T\bs y)_M\qquad\forall \bs x,\ \bs y \in \mathbb R^r.
$$
It is straightforward to verify that if both $A$ and $M$ are symmetric, then $AM$ is symmetric in $(\cdot,\cdot)_M$ and $(\cdot,\cdot)_{A^{-1}}$ inner products when the inner product is well defined.

For any symmetric matrices $A$ and $M$ with same sizes, we use $A\leq M$ to denote that $M - A$  is a positive semi-definite matrix and $A < M$ to denote that $M - A$ is a positive definite matrix. Orderings $\geq$ and $>$ are defined similarly.

The rest of the paper is organized as follows. In section 2, we present the convergence analysis of the BWY method and point out the generalization to the inexact Uzawa methods. In section 3, we construct approximate block factorization preconditioners for system \eqref{saddle_pro} and prove the so-called field-of-values equivalence which implies the convergence of GMRes method.

\section{Convergence analysis}
In this section, we will present our convergence analysis of the BWY method in detail and point out the main difference when apply to inexact Uzawa methods.

\subsection{Convergence of the BWY method}
The original saddle point system $\mathcal A$ can be factorized as
\begin{equation}\label{eq:exacfactorization}
\begin{pmatrix}
A & B^{\intercal} \\ B  & -C
\end{pmatrix}
=
\begin{pmatrix}
A  & 0 \\ B  & -S_A
\end{pmatrix}
\begin{pmatrix}
I  & A^{-1}B^{\intercal} \\ 0  &  I
\end{pmatrix},\quad\text{with}\quad S_A = BA^{-1}B^{\intercal} + C.
\end{equation}
As we mentioned in the introduction, the BWY method can be interpreted as inverting the block factorization $\hat{\mathcal A} = \mathcal L\, \mathcal H \, \mathcal U$ as an approximation of the factorization in \eqref{eq:exacfactorization}. 

Given an initial guess $(u^k,p^k)^{\intercal}$, we compute the residual $(r_u,r_p)^{\intercal}$ first, then replace matrices $A^{-1}$ and $S_A^{-1}$ by symmetric and positive definite matrices $R_A$ and $R_S$, respectively. We compute the correction $(e_u,e_p)^{\intercal}$ by inverting a lower triangulation system of $\mathcal L\, \mathcal H$, i.e.,
$$
\begin{pmatrix}
e_u \\ e_p
\end{pmatrix} = 
\begin{pmatrix}
R_A^{-1}  & 0   \\ B  &  -R_S^{-1}
\end{pmatrix}^{-1}
\begin{pmatrix}
r_u  \\ r_p
\end{pmatrix},
$$
and then transfer the correction $(e_u,e_p)^{\intercal}$ by the matrix $\mathcal U^{-1} = \begin{pmatrix}
I  &-R_AB^{\intercal} \\ 0  & I
\end{pmatrix}$. Combination of these two steps is equivalent to computing $\hat{\mathcal A}^{-1}$ by the block factorization \eqref{eq:Ahat}. 
%

The error operator for the BWY method is thus
$$
\mathcal E_{(u,p)} = \mathcal I - \hat{\mathcal A}^{-1} \mathcal A. 
$$
Although both $\hat{\mathcal A}$ and $\mathcal A$ are symmetric, they are not positive definite. 
Motivated by the transforming smoothers \cite{Brandt1979Multigrid,Wittum1990On},
we introduce a new variable $(v, q)^{\intercal}$ by change of variables
\begin{equation}\label{eq:vu}
\begin{pmatrix}
v \\ q
\end{pmatrix}
:= \mcal U
\begin{pmatrix}
u \\ p
\end{pmatrix}.
\end{equation}
The BWY iteration can be understood as the iteration of $(v,q)^{\intercal}$ in the form 
\begin{equation}
\label{eq:iter_meth}
\begin{pmatrix}
v^{k+1} \\ q^{k+1}
\end{pmatrix} =  
\begin{pmatrix}
v^k \\ q^k
\end{pmatrix} + \mathcal H^{-1} \mathcal L^{-1} \left(
\begin{pmatrix}
f \\ g   
\end{pmatrix} - 
\mathcal A \, \mathcal U^{-1} 
\begin{pmatrix}
v^k \\ q^k
\end{pmatrix}
\right),
\end{equation}
and then transfer to $(u,p)^{\intercal}$ as
\begin{equation}\label{eq:vuk}
\begin{pmatrix}
u^{k+1} \\ p^{k+1}
\end{pmatrix} = 
\mathcal U^{-1}
\begin{pmatrix}
v^{k+1} \\ q^{k+1}
\end{pmatrix}.
\end{equation}
The error operator for iteration \eqref{eq:iter_meth} for $(v, q)^{\intercal}$ variable is
\begin{equation}
\label{eq:err}
\mathcal E = \mathcal I - \mathcal H^{-1}\mathcal L^{-1}\mathcal A \, \mathcal U^{-1} = \mathcal U \mathcal E_{(u,p)} \mathcal U^{-1},
\end{equation}
which is the representation of $\mathcal E_{(u,p)}$ in the changed basis.

We aim to prove that iteration \eqref{eq:iter_meth} is convergent in $\|\cdot\|_{\mathcal D}$ norm, i.e., 
\begin{equation}
\|\mathcal E \bs x\|_{\mathcal D} \leq \rho \|\bs x\|_{\mathcal D},\qquad\forall\ \bs x = (v, q)^{\intercal} \in \mathbb R^{n+m},
\end{equation} 
where $\rho \in (0,1)$ is a constant independent of $\bs x$, and then transfer back to $(u, p)^{\intercal}$ by estimating the bound of $\|\, \mathcal U\|$ and $\|\, \mathcal U^{-1}\|$.

We first symmetrize the error operator $\mathcal E$.
\begin{lemma}\label{lem:F}
Let
$$
\mathcal E_{\mathcal A} = 
\begin{pmatrix}
R_A^{-1} - A & 0 \\ 0 & S - R_S^{-1}.
\end{pmatrix},
$$
and
$$
\mcal F = \mathcal D^{-1/2}\mathcal J \,\mathcal L^{-1} \mathcal E_{\mathcal A}\, \mathcal U^{-1} \mathcal J \, \mathcal D^{-1/2}.
$$
There holds
$$
\|\mathcal E \|_{\mathcal D} \leq \rho(\mathcal F).
$$
\end{lemma}
\begin{proof}
 By direct calculation, we have that
\begin{equation}\label{eq:E}
\mathcal E  = \mathcal H^{-1}\, \mathcal L^{-1} ( \hat{\mathcal A} - \mathcal A) \, \mathcal U^{-1}  = \mathcal H^{-1} \mathcal L^{-1} 
\mathcal E_{\mathcal A} \, \mathcal U^{-1}  = \mathcal D^{-1} \mathcal J \, \mathcal L^{-1}
\, \mathcal E_{\mathcal A} \, 
\mathcal U^{-1}.
\end{equation}
As $\mathcal D^{-1} \mathcal J \, \mathcal L^{-1} \mathcal E_{\mathcal A} \, \mathcal U^{-1} \mathcal J$ is symmetric in the inner product $(\cdot,\cdot)_{\mathcal D}$, we have
\begin{align*}
\|\mathcal E \bs x\|_{\mathcal D} & = \|\mathcal D^{-1}\mathcal J\, \mathcal L^{-1} \mathcal E_{\mathcal A} \, \mathcal U^{-1}\mathcal J\, \mathcal J \bs x\|_{\mathcal D} \\
&\leq \rho(\mathcal D^{-1}\mathcal J\, \mathcal L^{-1}  \mathcal E_{\mathcal A} \, \mathcal U^{-1}\mathcal J) \|\mathcal J \bs x\|_{\mathcal D}.
\end{align*}
Since $\mathcal D^{-1}\mathcal J\, \mathcal L^{-1}  \mathcal E_{\mathcal A} \, \mathcal U^{-1}\mathcal J$ is similar to $\mathcal F = \mathcal D^{-1/2}\mathcal J\, \mathcal L^{-1} \mathcal E_{\mathcal A} \, \mathcal U^{-1} \mathcal J\, \mathcal D^{-1/2}$ and $\|\mcal J \bs x\|_{\mathcal D} = \|\bs x\|_{\mathcal D}$, we get the desired inequality.
\end{proof}

In the sequel, we focus on the estimate of $\rho (\mathcal F)$. 
Let 
\begin{align*}
E_A &= I - R_A^{1/2}\, A\, R_A^{1/2},\quad
\bar B = R_S^{1/2}\, B\, R_A^{1/2}\, E_A,\\
\bar S  & = B\, \bar R_A \, B^{\intercal}+C,\quad
E_{\bar S} = I - R_S^{1/2}\, \bar S \, R_S^{1/2}.
\end{align*}
By direct calculation, we obtain
\begin{equation}\label{eq:F}
\mathcal F = 
\begin{pmatrix}
E_A  & \bar B^{\intercal} \\ \bar B  & -E_{\bar S}
\end{pmatrix}.
\end{equation}


\begin{remark}\rm
When $R_A^{-1} = A$, the BWY method can be interpreted as an iterative method for solving the Schur complement equation $(BA^{-1}B^{\intercal} + C) p = BA^{-1}f-g$. The matrix $\mathcal F$ reduces to $\begin{pmatrix} 0 & 0 \\ 0 & -E_{\bar S} \end{pmatrix}$, therefore $\rho(\mathcal F) = \rho(E_{\bar S}) = \rho(E_S)$, which implies that the convergence of the BWY method is only dependent on $\rho(E_{S})$. So in the rest of this section, we always assume that $R_A^{-1} >  A$. $\Box$
\end{remark}

We present the following inequalities from the assumption $R_A^{-1} > A$.
\begin{lemma}
Assume $R_A^{-1} > A$ and let $\delta = \rho(E_A)$. Then we have
\begin{align}
\label{eq:rel_R_A}
(1-\delta)R_A^{-1} &\leq A < R_A^{-1}, \\
\label{eq:rel_R_bar} R_A  & < \bar R_A \leq (1 + \delta)R_A, \\
\label{eq:S_S_bar} S  & < \bar S \leq (1 + \delta)S, \\
\label{eq:E_A} \bar B \, E_A^{-1} \bar B^{\intercal} & \leq \delta \ R_S^{1/2}\, S \, R_S^{1/2} < \delta \ R_S^{1/2}\, \bar S \, R_S^{1/2}.
\end{align}
\end{lemma}
\begin{proof}
Assumption $R_A^{-1} > A$ implies $\lambda_{\max}(R_A^{1/2}\, A\, R_A^{1/2}) < 1$. Therefore 
\begin{align*}
\rho (I - R_A^{1/2}\, A\, R_A^{1/2}) &= \max \left \{ \left |1 - \lambda_{\max}(R_A^{1/2}\, A\, R_A^{1/2}) \right |, \left |1 - \lambda_{\min}(R_A^{1/2}\, A\, R_A^{1/2})  \right |  \right \} \\
& = 1 - \lambda_{\min}(R_A^{1/2}\, A\, R_A^{1/2}).
\end{align*}
Then 
$$
R_A^{1/2}\, A\, R_A^{1/2} \geq \lambda_{\min}(R_A^{1/2}\, A\, R_A^{1/2}) I = (1-\delta) I
$$
implies the desired lower bound $A$. The upper bound of $A$ is the assumption.

To prove \eqref{eq:rel_R_bar}, we use the definition of $\bar R_A$ and bounds of $A$ in \eqref{eq:rel_R_A} to get
\begin{align*}
R_A < \bar R_A & = 2R_A - R_AA\, R_A \leq (1 + \delta) R_A.
\end{align*}
Inequality \eqref{eq:S_S_bar} is a direct consequence of \eqref{eq:rel_R_bar}.

The definition of $E_A$ and $\bar B$ imply that
\begin{align*}
\bar B \, E_A^{-1} \bar B^{\intercal} & = R_S^{1/2}\, B\, R_A^{1/2}\, E_AE_A^{-1}E_A\, R_A^{1/2}\, B^{\intercal} \, R_S^{1/2}  \\
&= R_S^{1/2}\, B\,(R_A - R_A\, A \,R_A)\,B^{\intercal} \, R_S^{1/2}.
\end{align*}
Inequality \eqref{eq:rel_R_A} implies that
$$
R_A - R_AA\, R_A \leq \delta R_A.
$$
Therefore, it holds
\begin{align*}
\bar B \, E_A^{-1} \bar B^{\intercal} & \leq \delta R_S^{1/2} B\, R_A \, B^{\intercal} \, R_S^{1/2} 
 \leq \delta R_S^{1/2} (C + B R_A B^{\intercal}) R_S^{1/2},
\end{align*}
which implies inequalities in \eqref{eq:E_A}.
%
\end{proof}

Similarly we have the following inequality from the assumption $R_S^{-1}\geq \bar S$.
\begin{lemma}
 Assume $R_S^{-1}\geq \bar S$ and let $\bar \gamma = \rho (E_{\bar S})$. Then we have
 \begin{equation}\label{eq:RSR}
R_S^{1/2}\, \bar S \, R_S^{1/2} \geq (1 - \bar \gamma ) I.
\end{equation}
\end{lemma}

We now rescale $\mathcal F$ to further simplify its formulation.
\begin{lemma}\label{lem:rho_F}
Assume $R_A^{-1} > A$ and let $\delta = \rho(E_A)$. Denote 
$$
\mathcal M = \begin{pmatrix}
\delta^{-1/2}E_A^{1/2}  & 0 \\ 0 & I
\end{pmatrix}\text{ and } \mathcal T = \begin{pmatrix}
\delta \, I & \delta^{1/2} E_A^{-1/2}\bar B^{\intercal} \\
\delta^{1/2} \bar B \, E_A^{-1/2}  & -E_{\bar S}
\end{pmatrix},
$$ 
then $\mathcal F = \mathcal M \mathcal T \mathcal M$ and thus
$$
\rho(\mathcal F) \leq \rho(\mathcal T).
$$
\end{lemma} 
\begin{proof}
It is straightforward to verify that $\mathcal F = \mathcal M \mathcal T \mathcal M$. Since $\mathcal F$, $\mathcal M$ and $\mathcal T$ are symmetric, we have
$$
\rho(\mathcal F) = \|\mathcal F\| \leq \|\mathcal M\|^2 \|\mathcal T\| \leq \rho(\mathcal T).
$$
\end{proof}

In the rest of this section, we will focus on the estimate of $\rho(\mathcal T)$. 
\begin{lemma}\label{lem:M}
Assume that $R_A^{-1} > A, R_S^{-1} \geq \bar S$
and 
$\delta < \frac{\sqrt{5} - 1}{2}$,
then we have
$$
\rho(\mathcal T) \leq \rho_1<1,
$$
where 
$$\rho_1 = \max \left\{\frac{-(\delta-\bar \gamma) + \sqrt{(\delta-\bar \gamma)^2 + 4(\delta^2(1-\bar \gamma) + \delta\bar \gamma )}}{2}, \frac{\sqrt{5}+1}{2}\delta\right\}.$$
\end{lemma} 
\begin{proof}
Let $\lambda$ be any eigenvalue of $\mathcal T$ and $(v,q)^{\intercal}$ 
be the corresponding eigenvector. As $\mathcal T$ is symmetric, $\lambda \in \mathbb R$, by definition, it holds 
\begin{align}
\label{eq:eigensystem1} \delta \,v + \delta^{1/2}\, E_A^{-1/2}\bar B^{\intercal} \,q &= \lambda\, v, \\
\label{eq:eigensystem2} \delta^{1/2}\bar B \, E_A^{-1/2} v - E_{\bar S} \, q &= \lambda \,q.
\end{align}

If $\lambda  \in [\delta - \delta^2, \delta]$, we already have $|\lambda| \leq \delta < \rho_1$. Therefore,
we only consider cases that $\lambda \in (-\infty,\delta-\delta^2)$ and $\lambda\in (\delta,\infty)$. In these cases, we can solve $v$ in \eqref{eq:eigensystem1}
$$
v = (\lambda - \delta)^{-1}\delta^{1/2}\, E_A^{-1/2}\bar B^{\intercal} q,
$$
and substitute into \eqref{eq:eigensystem2} to get the equation
\begin{equation}\label{eq:eigen}
\frac{\delta}{\lambda - \delta} \bar B \, E_A^{-1} \bar B^{\intercal} q - E_{\bar S} q = \lambda \, q.
\end{equation}

Let 
$$
\Phi(\mu) = \frac{\delta}{\mu - \delta} \bar B \, E_A^{-1} \bar B^{\intercal} - E_{\bar S} - \mu I = \frac{\delta}{\mu - \delta} \bar B \, E_A^{-1} \bar B^{\intercal} + R_S^{1/2}\, \bar S \, R_S^{1/2} - (1 + \mu )I,
$$
and define $$h(\mu) = (\Phi(\mu) q,q).$$

Equation \eqref{eq:eigen} implies that $\lambda$ is a zero point of  the function $h(\mu)$. Since $$h'(\mu) = -(\delta/(\mu - \delta)^2 \bar B \, E_A^{-1} \bar B^{\intercal} q, q ) - (q,q) < 0$$
 holds for all $\mu \neq \delta$, the function $h(\mu)$ is strictly decreasing in intervals $(-\infty,\delta - \delta^2)$ and $(\delta,\infty)$. 

Consider the case $\mu < \delta-\delta^2 < \delta$. The coefficient $\delta/(\mu - \delta) < 0$ and $\left( \frac{\delta^2}{\mu - \delta}+1\right) > 0$. Using \eqref{eq:E_A} and \eqref{eq:RSR}, we have the inequality
\begin{equation}\label{eq:hg0}
\begin{aligned}
\Phi(\mu) & > \left( \frac{\delta^2}{\mu - \delta} + 1 \right) R_S^{1/2}\, \bar S \, R_S^{1/2} - (1 + \mu) I \\
& \geq \left [\left( \frac{\delta^2}{\mu - \delta}+1\right)(1-\bar \gamma) - (1+\mu) \right ]I.
\end{aligned}
\end{equation}
Therefore, if 
\begin{equation}
\label{eq:mu1}
\mu < \frac{(\delta-\bar \gamma) - \sqrt{(\delta-\bar \gamma)^2 + 4(\delta^2(1-\bar \gamma) + \delta\bar \gamma )}}{2} := \mu_1(\delta, \bar \gamma)< 0,
\end{equation}
it holds
$$
(\delta^2 + \mu - \delta)(1 - \bar \gamma) < (1+\mu)(\mu - \delta),
$$
which implies $h(\mu) >0$.
%
%
As $h(\mu)$ is strictly decreasing, all roots of $h(\mu) = 0$ in $(-\infty,\delta-\delta^2)$ are greater than or equal to $\mu_1$, which implies that $\lambda \in [\mu_1,\delta - \delta^2)$. Therefore $$|\lambda| \leq \max \left \{|\mu_1|, \delta - \delta^2 \right \} \leq \rho_1.$$
The bound $|\mu_1|< 1$ is from the fact $\bar \gamma < 1$ and $\delta < 1$.

In the case $\mu > \delta$, the coefficient $\delta/(\mu - \delta) > 0$. We use  inequality \eqref{eq:E_A} and the assumption $R_S^{-1} \geq \bar S$ to get 
\begin{equation}\label{eq:hs0}
\Phi(\mu) < \left( \frac{\delta^2}{\mu - \delta} + 1 \right) R_S^{1/2}\, \bar S \, R_S^{1/2} - (1 + \mu) I \leq \left( \frac{\delta^2}{\mu - \delta} - \mu \right ) I.
\end{equation}
Therefore, if $\mu > \frac{1 + \sqrt{5}}{2} \delta$, then 
$$
\delta^2 + \delta \mu - \mu^2 < 0,
$$
and consequently $h(\mu) < 0$.
We conclude that the roots of $h(\mu) = 0$ in the interval $(\delta,\infty)$ should be less than or equal to $\frac{1 +\sqrt{5}}{2}\delta$, which means $\delta <\lambda \leq \frac{\sqrt{5}+1}{2}\delta \leq \rho_1$.

Combine these cases together, we obtain the desired bound for $\rho(\mathcal T)$. 
\end{proof}

\begin{remark}\rm
When $R_A^{-1} = A$, i.e., $\delta = 0$, we have $\rho_1 = \bar\gamma = \gamma$, which is consistent with the convergence results of Uzawa methods.  $\Box$
\end{remark}

\begin{remark}\label{rm:otherversion}\rm
We can obtain similar estimate by replacing the assumption $R_S^{-1} \geq \bar S$
and $\delta < \frac{\sqrt{5} - 1}{2}$ by $R_S^{-1} \geq S$
and $\delta < \frac{1}{2}$. Indeed for the case $\mu < \delta-\delta^2$, the calculation is identical by using $S$ and $\gamma$. For the case $\mu > \delta$, using the relation $\bar S\leq (1+\delta)S$, inequality \eqref{eq:hs0} becomes
$$
\Phi(\mu) \leq \left( \frac{\delta^2}{\mu - \delta} + 1 + \delta \right) R_S^{1/2}\, S \, R_S^{1/2} - (1 + \mu) I \leq \left( \frac{\delta^2}{\mu - \delta} - (\mu -\delta)\right ) I.
$$ 
Therefore $\lambda \leq 2\delta < 1$ if $\delta < 1/2$. The contraction factor becomes
$$
\tilde\rho_1 = \max \left \{ |\mu_1(\delta,\gamma)|, 2\delta \right \}.
$$
The assumption $R_S^{-1} \geq S$ is weaker than $R_S^{-1} \geq \bar S$ as $\bar S \geq S$. The upper bound for $\delta$ is, however, more tight and the contraction rate $\tilde \rho_1$ is slightly larger than $\rho_1$. $\Box$
\end{remark}

To transfer back to the original variable, we need to estimate the norm of the transformation.
\begin{lemma}\label{lem:U_norm}
Assume that $R_A^{-1} > A$ and $R_S^{-1} \geq S$. Then for any $\bs x=(v,q)^\intercal \in \mathbb R^{n+m}$ with $v \in \mathbb R^n$ and $q\in \mathbb R^m$, there hold
\begin{align}
\label{eq:u_up}
\max\{\, \|\mathcal U\bs x\|_{\mathcal D}^2,\, \|\mathcal U^{-1}\bs x\|^2_{\mathcal D} \} \leq 3\|\bs x\|_{\mathcal D}^2.
\end{align}
\end{lemma}
\begin{proof}
By simple calculation, it holds
\begin{align*}
\|\mathcal U\bs x\|^2_{\mathcal D} & = \|v + R_AB^\intercal q\|^2_{R_A^{-1}} + \|q\|^2_{R_S^{-1}}  \\
& \leq 2\|v\|^2_{R_A^{-1}} + 2\|R_AB^\intercal q\|_{R_A^{-1}}^2 + \|q\|^2_{R_S^{-1}}  \\
&\leq 3\|\bs x\|^2_{\mathcal D}.
\end{align*}
Here in the second step, we have used Cauchy-Schwarz inequality and in the third step, we use the assumption $S \leq R_S^{-1}$. Proofs of the bound for $\|\, \mathcal U^{-1}\bs x\|_{\mathcal D}$ is almost identical. 
\end{proof}

We summarize the above estimates as the following theorem. 
\begin{theorem}\label{the:conv}
 Assume that $R_A^{-1} > A, R_S^{-1} \geq \bar S$ and
 $
 \delta < \frac{\sqrt{5} - 1}{2},
 $
 then the BWY iteration is convergent
 \begin{align}\label{eq:error}
 \|u - u^k\|_{R_A^{-1}}^2 + \|p-p^k\|_{R_S^{-1}}^2 \leq 9 \rho_1^{2k}\left( \|u-u^0\|_{R_A^{-1}}^{2} + \|p-p^0\|_{R_S^{-1}}^2 \right),
 \end{align} 
where $\rho_1 = \rho_1(\delta, \bar \gamma) \in (0,1)$ is given in Lemma \ref{lem:M}.
\end{theorem}
\begin{proof}
Using Lemma \ref{lem:F}, \ref{lem:rho_F} and \ref{lem:M}, we obtain that
$$
 \left\|
\begin{pmatrix}
v \\ q
\end{pmatrix}
-
\begin{pmatrix}
v^k \\ q^k
\end{pmatrix}
 \right\|_{\mathcal D} \leq \rho_1^k
  \left\|
\begin{pmatrix}
v \\ q
\end{pmatrix}
-
\begin{pmatrix}
v^0 \\ q^0
\end{pmatrix}
 \right\|_{\mathcal D}.
$$ 
Therefore, by Lemma \ref{lem:U_norm}, it holds
\begin{align*}
 \left\|
\begin{pmatrix}
u \\ p
\end{pmatrix}
-
\begin{pmatrix}
u^k \\ p^k
\end{pmatrix}
 \right\|_{\mathcal D}^2 & \leq 
 3 \left\|
\begin{pmatrix}
v \\ q
\end{pmatrix}
-
\begin{pmatrix}
v^k \\ q^k
\end{pmatrix}
 \right\|_{\mathcal D}^2 \leq 3\rho_1^{2k}
 \left\|
\begin{pmatrix}
v \\ q
\end{pmatrix}
-
\begin{pmatrix}
v^0 \\ q^0
\end{pmatrix}
 \right\|_{\mathcal D}^2 \\
 &\leq 9\rho_1^{2k} \left\|
\begin{pmatrix}
u \\ p
\end{pmatrix}
-
\begin{pmatrix}
u^0 \\ p^0
\end{pmatrix}
 \right\|_{\mathcal D}^2 .
\end{align*}
The desired result then follows.
\end{proof}
\begin{remark}\rm
 According to Remark \ref{rm:otherversion}, we will get a similar convergence with assumptions
 $R_A^{-1} > A, R_S^{-1} \geq S$ and
 $ \delta < \frac{1}{2},$
 and the rate is changed to $\tilde \rho_1(\delta, \gamma)$ given in Remark \ref{rm:otherversion}.
\end{remark}

\subsection{Convergence of the symmetrized inexact Uzawa method}
The symmetrized inexact Uzawa method \eqref{SIU_1}-\eqref{SIU_3} can be understood as the iteration \eqref{eq:iter_meth} for $(v, q)^{\intercal}$ with $\mathcal H$ being replaced by $\bar{\mathcal H}: =  
\begin{pmatrix}
\bar R_A^{-1} & 0 \\ 0 & -R_S^{-1}
\end{pmatrix}$ and then transfer to $(u,p)$ using \eqref{eq:vuk}.  The error operator of the renewed iteration \eqref{eq:iter_meth} is 
$$
\bar{\mathcal E} = I - \bar{\mathcal H}^{-1}\mathcal L^{-1} \mathcal A \, \mathcal U^{-1}.
$$
Convergence analysis of this iteration is similar. So we only outline the key formulae in the calculation. Let
\begin{align*}
\bar{\mathcal E}_A &= \begin{pmatrix}
\bar R_A^{-1} - A   &  (\bar R_A^{-1}R_A - I)B^\intercal \\ B(R_A\bar R_A^{-1} - I)  &  B\, R_A\, \bar R_A^{-1} R_A \,  B^\intercal + C - R_S^{-1}
\end{pmatrix},\\
\bar{\mathcal F} &= \bar{\mathcal D}^{-1/2}\mathcal J \mathcal L^{-1} \bar{\mathcal E}_{\mathcal A} \, \mathcal U^{-1} \mathcal J \bar{\mathcal D}^{-1/2} = \begin{pmatrix}
\bar E_A   &  \hat B^\intercal \\ \hat B  & - E_{\bar S}
\end{pmatrix},\\ 
\bar{\mathcal M} & = \begin{pmatrix}
\delta^{-1}\bar E_A^{1/2}  &  0 \\ 0  & I
\end{pmatrix},\qquad
\bar{\mathcal T} = 
\begin{pmatrix}
\delta^2 I & \delta \,  \bar E_A^{-1/2} \hat B^\intercal \\ 
\delta \, \hat B\,  \bar E_A^{-1/2}   &   R_S^{1/2}\,  \bar S \,  R_S^{1/2} - I
\end{pmatrix},
\end{align*}
with
$$
\bar{\mathcal D} = \begin{pmatrix}  
\bar R_A^{-1}  & 0 \\ 0 & R_S^{-1}
\end{pmatrix},
\qquad \bar E_A = I - \bar R_A^{1/2}A \bar R_A^{1/2},\qquad \hat B = R_S^{1/2} B(I - R_AA) \bar R_A^{1/2}.
$$
Similar to the proof of \eqref{eq:E_A}, under the assumption $R_A^{-1} > A$, we can prove that
\begin{equation}
\label{eq:E_A_bar} \hat B \bar E_A^{-1} \hat B^T  \leq  R_S^{1/2}\bar S R_S^{1/2},
\end{equation}
and
$$
\bar{\mathcal F} = \bar{\mathcal M} \bar{\mathcal T} \bar{\mathcal M}, \quad \rho( \bar{\mathcal F}) \leq \rho ( \bar{\mathcal T})
$$
Computing as in the proof of Lemma \ref{lem:M}, we get 
$$
\bar\Phi(\mu) =  \frac{\delta^2}{\mu - \delta^2} \hat B \bar E_A^{-1} \hat B^T + R_S^{1/2}\bar S R_S^{1/2} - (1 + \mu )I,
$$
and
$$
\bar h(\mu) = (\bar\Phi(\mu) q,q).
$$
By the same line as the proof of Theorem \ref{the:conv} for intervals $(-\infty,0)$ and $(\delta^2,+\infty)$ and using Lemma \ref{lem:U_norm}, we will get the following convergence result for the symmetrized inexact Uzawa method.
\begin{theorem}\label{the:conv_2}
 Assume that $R_A^{-1} >  A, R_S^{-1} \geq \bar S$ and
 $
 \delta < \frac{\sqrt{2} }{2},
 $
 then the symmetrized Uzawa method \eqref{SIU_1}-\eqref{SIU_3} is convergent with
 \begin{align}\label{eq:error1}
 \|u - u^k\|_{\bar R_A^{-1}}^2 + \|p-p^k\|_{R_S^{-1}}^2 \leq 9\rho_2^{2k}\left( \|u-u^0\|_{\bar R_A^{-1}}^{2} + \|p-p^0\|_{R_S^{-1}}^2 \right),
 \end{align} 
 where 
\begin{equation}\label{eq:rho2}
 \rho_2 = \max \left\{\ \frac{(\delta^2-\bar \gamma) - \sqrt{(\delta^2-\bar \gamma)^2 + 4\delta^2}}{2}, \frac{\delta^2 ( 1 + \sqrt{\delta^2 + 4})}{2}\right\}.
\end{equation}
\end{theorem}
\begin{remark}\rm
Again we can obtain similar estimate by using assumptions $R_A^{-1} \geq A$ $R_S^{-1} \geq S$ and $\delta < \frac{1}{2}$ with a different contraction factor 
$$
\tilde\rho_2 = \max \left\{ \frac{(\gamma - \delta^2) + \sqrt{(\gamma - \delta^2)^2 + 4\left( \delta^3(1 - \gamma) + \delta^2 \right)}}{2},\ \frac{1+\delta + \sqrt{\delta^2 + 2\delta + 5}}{2} \delta \right\}.
$$
\end{remark}

\subsection{Convergence of the IUM method}
We consider the IUM method \eqref{IUM_1}-\eqref{IUM_2}. Unfolding equation \eqref{IUM_1} and writting two consecutive steps of IUM together, we get
\begin{align}
\label{IUM_M1}u^{k+1/2} & = u^k + R_A(f - Au^k - B^\intercal p^k), \\
\label{IUM_M2}u^{k+1} & = u^{k+1/2} + R_A(f - Au^{k+1/2} - B^\intercal p^k), \\
\label{IUM_M3}p^{k+1} & = p^k - R_S(g - Bu^{k+1} + Cp^k) ,\\
\label{IUM_M4}u^{k+3/2} & = u^{k+1} + R_A(f - Au^{k + 1} - B^\intercal p^{k+1}), \\
\label{IUM_M5}u^{k+2} & = u^{k+3/2} + R_A(f - Au^{k+3/2} - B^\intercal p^{k+1}), \\
\label{IUM_M6}p^{k+2} & = p^{k+1} - R_S(g - Bu^{k+2} + Cp^{k+1}) .
\end{align}
We can regroup the iterations and view \eqref{IUM_M2}-\eqref{IUM_M4} as one iteration step of the symmetrized inexact Uzawa method \eqref{SIU_1}-\eqref{SIU_3}, i.e., $(u^{k+3/2},p^{k+1})^\intercal$ is generated by the symmetrized Uzawa method \eqref{SIU_1}-\eqref{SIU_3} from $(u^{k+1/2},p^{k})^\intercal$ and thus the sequence $\{(u^{k+1/2},p^{k})^\intercal \}$ is convergent with usual assumptions. To prove the convergence of $\{(u^{k+1},p^{k+1})^\intercal \}$, we need the following preparation. 

\begin{lemma}
Assume that $R_A^{-1} > A$. Then 
$$
\| I - R_AA\|_{\bar R_A^{-1}} < 1.
$$ 
\end{lemma}
\begin{proof}
By definition 
$$
(I - R_AA)^2 = I - \bar{R}_AA
$$
is symmetric in the inner product $(\cdot,\cdot)_{\bar{R}_A^{-1}}$, so is $I - R_AA$. Therefore $\| I - R_AA\|_{\bar R_A^{-1}} = \rho (I - R_AA) < 1$.
\end{proof}

\begin{lemma}
Assume that $R_S^{-1} \geq \bar S$, then we have the inequality 
\begin{equation}
\label{eq:RSB}
B^\intercal R_S B \leq \bar R_A^{-1}.
\end{equation}
\end{lemma}
\begin{proof}
For any $v \in \mathbb R^n$, let $q = R_SBv \in \mathbb R^m$. 
Then
$$
(B^\intercal R_S Bv,v) = (R_SB v, B v) = (R_S^{-1}q,q) = \|q\|_{R_S^{-1}}^2.
$$
On the other hand,
\begin{align*}
\|q\|_{R_S^{-1}} & = \sup\limits_{p \in \mathbb R^m} \frac{(q,p)_{R_S^{-1}}}{\|p\|_{R_S^{-1}}}  = \sup\limits_{p \in \mathbb R^m} \frac{(Bv,p)}{\|p\|_{R_S^{-1}}} = \sup\limits_{p \in \mathbb R^m} \frac{(v,B^\intercal p)}{\|p\|_{R_S^{-1}}} \\
& \leq \sup\limits_{p \in \mathbb R^m} \frac{\|v\|_{\bar R_A^{-1}} \|B^\intercal p\|_{\bar R_A}}{\|p\|_{R_S^{-1}}} \leq \|v\|_{\bar R_A^{-1}},
\end{align*}
where in the last step, we have used
$$
B\bar R_A B^\intercal \leq B\bar R_A B^\intercal +C = \bar S \leq R_S^{-1},
$$
to get $\|p\|_{R_S^{-1}} \geq \|B^\intercal p\|_{\bar R_A}$.
The desired result then follows.
\end{proof}

We can now bound the error of $(u^{k+1},p^{k+1})$ by that of $(u^{k+1/2},p^k)$.
\begin{lemma}\label{lem:upk}
Assume that $R_A^{-1} >  A$ and $R_S^{-1} \geq \bar S$, there hold
 \begin{align*}
 \|u - u^{k+1}\|_{\bar R_A^{-1}}^2 & \leq 2 \left( \|u - u^{k+1/2}\|_{\bar R_A^{-1}}^2 + \|p - p^k\|_{R_S^{-1}}^2 \right), \\
 \|p - p^{k+1}\|_{R_S^{-1}}^2 & \leq 2 \left( \|u - u^{k+1/2}\|_{\bar R_A^{-1}}^2 + \|p - p^k\|_{R_S^{-1}}^2 \right).
 \end{align*}
\end{lemma}
\begin{proof}
Equation \eqref{IUM_M2} implies that
\begin{equation}\label{eq:uk}
u - u^{k+1} = (I - R_AA)(u - u^{k+1/2}) - R_AB^\intercal (p - p^k),
\end{equation}
therefore,
\begin{align*}
\|u - u^{k+1}\|_{\bar R_A^{-1}}^2 & \leq 2 \left( \|(I - R_AA)(u - u^{k+1/2})\|^2_{\bar R_A^{-1}}  + \|R_A B^\intercal (p - p^k)\|_{\bar R_A^{-1}}^2 \right) 
\\
& \leq 2 \left(\|u - u^{k+1/2}\|^2_{\bar R_A^{-1}}  + \|p - p^k\|_{R_S^{-1}}^2 \right).
\end{align*}
Here in the second inequality we have used 
$$
BR_A\bar R_A^{-1}R_AB^\intercal \leq B\bar R_AB^\intercal \leq \bar S \leq R_S^{-1}.
$$
Equations \eqref{IUM_M3} and \eqref{eq:uk} imply that 
$$
p - p^{k+1} = (I - R_SS)(p - p^k) +R_SB(I - R_AA)(u - u^{k+1/2}),
$$
Then, it holds
\begin{align*}
\|p - p^{k+1}\|_{R_S^{-1}}^2 & \leq 2 \left( \|(I - R_SS)(p - p^k)\|_{R_S^{-1}}^2 + \|R_SB(I - R_AA)(u - u^{k+1/2})\|_{R_S^{-1}}^2 \right) \\
& \leq 2 \left( \|p - p^k\|_{R_S^{-1}}^2 + \|(I - R_AA)(u - u^{k+1/2})\|_{\bar R_A^{-1}}^2 \right) \\
& \leq 2 \left( \|p - p^k\|_{R_S^{-1}}^2 + \|u - u^{k+1/2}\|_{\bar R_A^{-1}}^2 \right).
\end{align*}
Here, in the second inequality, we have used inequality \eqref{eq:RSB}.
\end{proof}

Theorem \ref{the:conv_2} and Lemma \ref{lem:upk} imply the following result. 
\begin{theorem}\label{the:conv_3}
 Assume that $R_A^{-1} >  A, R_S^{-1} \geq \bar S$ and
 $\delta < \frac{\sqrt{2} }{2},$
 then the IUM method \eqref{IUM_1}-\eqref{IUM_2} is convergent with
 \begin{align}\label{eq:error2}
 \|u - u^k\|_{\bar R_A^{-1}}^2 + \|p-p^k\|_{R_S^{-1}}^2 \leq  36\rho_2^{2k}\left( \| u-u^{1/2} \|_{\bar R_A^{-1}}^{2} + \|p-p^0\|_{R_S^{-1}}^2 \right),
\end{align} 
where $\rho_2$ is given in \eqref{eq:rho2}.
\end{theorem}

\section{Approximate Block Factorization Preconditioner for GMRes}

In this section, we will construct approximate block factorization preconditioners for system \eqref{saddle_pro}. 
Recall the block decomposition
\begin{equation}
\label{eq:block_decom}
\hat{\mathcal A} 
= \begin{pmatrix}
R_A^{-1}   &  B^{\intercal}  \\ B  &  B\, R_AB^\intercal - R_S^{-1}
\end{pmatrix} = 
\mathcal L\, \mathcal H \, \mathcal U
\end{equation}
The action $(\mathcal L\, \mathcal H)^{-1}$ can be understood as a block Gauss-Seidel iteration and $\mathcal U^{-1}$ is a distribution or change of variables. 

We consider two preconditioners for $\mathcal A$.  Define the operator $\mathcal G: \mathbb R^{n+m} \rightarrow  \mathbb R^{n+m}$ as
\begin{equation}\label{eq:Matrix_Pre_DGS_curl}
\mathcal G = \mathcal U^{-1}(\mathcal L\, \mathcal H )^{-1}=
\begin{pmatrix}
I  & -R_A B^{\intercal} \\ 0  &   I
\end{pmatrix}
\begin{pmatrix}
R_A^{-1}  & 0 \\ B  & R_S^{-1}
\end{pmatrix}^{-1},
\end{equation}
and use $\mathcal G$ as a left preconditioner for $\mcal A$. The preconditioned system is 
\begin{equation}\label{eq:pre_cond_1}
\mathcal G\mathcal A = \mathcal U^{-1}(\mathcal L\, \mathcal H )^{-1} \mathcal A.
\end{equation}
Another preconditioned system is defined as
\begin{equation}\label{eq:pre_cond_2}
(\mathcal L\, \mathcal H)^{-1}\mathcal A\, \mathcal U^{-1}.
\end{equation}
That is we apply a left preconditioner $(\mathcal L\, \mathcal H) ^{-1}$  and a right preconditioner $\mathcal U^{-1}$ to $\mathcal A$. 
Obviously these two preconditioners will have identical computation cost in each step. 

As non-SPD operators used in the preconditioners, we shall apply the generalized minimal residual method (GMRes) to $\mcal G\mcal A$ or $(\mathcal L\, \mathcal H)^{-1}\mathcal A\, \mathcal U^{-1}$. 

To prove convergence of GMRes, a bound of the so-called field-of-values- (FOV-) equivalence~\cite{Elman1982,Loghin;Wathen2004} is needed. To do so, we impose the following assumptions on the spectrum: there exist constants $0< \mu_1\leq \mu_2 <2$ and $0<\kappa_1 \leq \kappa_2$ so that
\begin{align}\label{eq:ass_equ}
\mu_1 R_A^{-1}  &\leq  A \leq \mu_2 R_A^{-1},\\
\label{eq:schur_equ} \kappa_1 R_S^{-1} &\leq  S \leq \kappa_2 R_S^{-1}.
\end{align}
As we shall show in a moment, assumption \eqref{eq:ass_equ} is equivalent to $\delta = \rho (I - R_A A)< 1$. Here we introduce constants $ \mu_1, \mu_2$ as they will appear in the estimate. For Schur complement, assumption \eqref{eq:schur_equ} always holds with $\kappa_1 = \lambda_{\min}(R_SS)$ and $\kappa_2= \lambda_{\max}(R_SS)$. 

Let $E = I - R_A A$. Condition \eqref{eq:ass_equ} implies that the iterative method using $R_A$ to solve $A^{-1}$ is convergent. More precisely, we have the following spectrum estimates.
\begin{lemma}\label{lem:spec_E}
Let $\delta = \rho (I - R_A A)$. Assume \eqref{eq:ass_equ} holds, then
\begin{align}
\|E\|_{R_A^{-1}} &= \delta <1,\\
\label{eq:ERA} E^{\intercal} R_A^{-1} E & \leq \delta^2 R_A^{-1}, \\
\label{eq:ET} E \, R_A \, E^{\intercal} & \leq \delta^2 R_A.
\end{align}
\end{lemma}
\begin{proof}
Equation \eqref{eq:ass_equ} implies that $\lambda_{\min}( R_AA) \geq \mu_1$ and $\lambda_{\max}(R_AA) \leq \mu_2$. Since $E$ is symmetric in the inner product $(\cdot,\cdot)_{R_A^{-1}}$, it holds
$$
\|E\|_{R_A^{-1}} = \max\{|1 - \lambda_{\min}( R_AA)|, |1 -\lambda_{\max}( R_AA)| \} \leq \max \{ |1- \mu_1 |, |1- \mu_2|\} < 1.
$$
For any $u \in \mathbb R^n$, we have
$$
\|Eu\|_{R_A^{-1}} \leq \|E\|_{R_A^{-1}}\|u\|_{R_A^{-1}} \leq \delta\|u\|_{R_A^{-1}},
$$
which implies \eqref{eq:ERA}. Similarly, using the fact that $E^{\intercal}$ is symmetric in the inner product $(\cdot,\cdot)_{R_A}$, we can prove \eqref{eq:ET}.
\end{proof}

The following lemma gives a bound of $B^{\intercal}S^{-1}B$.
\begin{lemma}\label{lm:BAB}
We have the inequality
\begin{equation}\label{BAB}
B^{\intercal}S^{-1} B \leq R_A^{-1}.
\end{equation}
\end{lemma}
\begin{proof}
For any $v \in \mathbb R^n$, let $q = S^{-1} B v\in \mathbb R^m$. Then
\begin{align*}
(  B^{\intercal} S^{-1} B v, v) =(  S^{-1} B v, B v) = ( S q, q) = \|q\|_{S}^{2}.
\end{align*}
Note that
\begin{align*}
\|q\|_{S} &= \sup_{p\in \mathbb R^m}\dfrac{( q,p) _{S}}{\|p\|_{S}}
= \sup_{p\in \mathbb R^m}\dfrac{( B v,p)}{\|p\|_{S}}
= \sup_{p\in \mathbb R^m}\dfrac{( v,B^{\intercal}p)}{\|p\|_{S}} \\
&\leq \sup_{p\in \mathbb R^m}\dfrac{\|v\|_{R_A^{-1}}\|B^{\intercal}p\|_{R_A}}{\|p\|_{S}}
\leq \|v\|_{R_A^{-1}}.
\end{align*}
In the last step, we have used the definition $S = B\, R_AB^{\intercal}+C$ which implies $\|B^{\intercal}p\|_{R_A}\leq \|p\|_{S}$.
The desired result \eqref{BAB} then follows easily. 
\end{proof}

We estimate the field of values equivalence of preconditioned system $(\mathcal L\, \mathcal H) ^{-1}\mathcal A\, \mathcal U^{-1}$.

\begin{lemma}\label{lem:fiels_pre}
Assume \eqref{eq:ass_equ} and \eqref{eq:schur_equ} hold,
then for all non-zero element $\bs x = (u, p)^{\intercal} \in \mathbb R^n \times \mathbb R^m$, it holds
\begin{align}\label{eq:FoE1}
& \frac{( \mathcal H^{-1} \mathcal L^{-1} \mathcal A\, \mathcal U^{-1} \bs x, \bs x )_{\mathcal D}}{(\bs x, \bs x )_{\mathcal D}} \geq \gamma,\\
\label{eq:FoE2} &\frac{\|\mathcal H^{-1} \mathcal L^{-1}\mathcal A\, \mathcal U^{-1} \bs x\|_{\mathcal D}}{\|\bs x\|_{\mathcal D}} \leq \Gamma,
\end{align}  
where $\gamma = \min\left\{\mu_1, \min\{2 - \mu_2, 1 \} \kappa_1  \right\}$ and $\Gamma = \left( 2 \max\left\{ \mu^2_2 + \kappa_2\delta^2, 2\kappa_2^2(1 + \delta^2) + \kappa_2\delta^2\right\}\right)^{1/2}$. 
\end{lemma}
\begin{proof}
By direct calculation, we have
$$
\mathcal H^{-1} \mathcal L^{-1}\mathcal A\, \mathcal U^{-1} = \mathcal D^{-1}
\begin{pmatrix}
A  &  E^{\intercal}B^{\intercal}  \\ -BE  & S + BER_AB^{\intercal} 
\end{pmatrix}.
$$
Note that the sign change of the second row if we replace $\mathcal H$ by $\mathcal D$. Then
we have
\begin{align*}
( \mathcal H^{-1} \mathcal L^{-1}\mathcal A\, \mathcal U^{-1}\bs x, \bs x)_{\mathcal D} & = \|u\|_{A}^2 + ( (S + BER_AB^{\intercal} ) p, p ). 
\end{align*}
The first term is easy: $\|u\|_{A}^2 \geq \mu_1 \|u\|_{R_A^{-1}}$ by assumption. 
By the definition of $S$ and the assumption \eqref{eq:ass_equ}, we have
\begin{align*}
S + BER_AB^{\intercal} & = S + B\, R_A \, B^{\intercal} - B\, R_A A R_A B^{\intercal} \\
& =  C + B(2R_A- R_A A R_A)B^{\intercal} \\
& \geq  C + (2 - \mu_2)B R_A B^{\intercal} \\
& \geq \min\{2 - \mu_2, 1 \} S \\
& \geq \min\{2 - \mu_2, 1 \}\kappa_1 R_S^{-1}.
\end{align*}
Thus, inequality \eqref{eq:FoE1} follows with $\gamma = \min\left\{\mu_1, \min\{2 - \mu_2, 1 \} \kappa_1  \right\}$.

To prove the upper bound, we split it as
\begin{align}\label{eq:Uper_bound}
\begin{split}
\|\mathcal H^{-1} \mathcal L^{-1}\mathcal A\, \mathcal U^{-1} \bs x\|_{\mathcal D} & \leq  \left\|  
\mathcal D^{-1} \begin{pmatrix}
A  & 0 \\ 0 & S + BER_AB^{\intercal}
\end{pmatrix} \bs x
\right\|_{\mathcal D} \\
&\quad  + \left\| \mathcal D^{-1}
\begin{pmatrix}
0  & -E^{\intercal}B^{\intercal} \\ BE & 0
\end{pmatrix}\bs x
 \right\|_{\mathcal D}. 
\end{split}
\end{align}
We first estimate the diagonal part of \eqref{eq:Uper_bound}. For $u$ part, it is an easy consequence of 
\begin{align*}
( R_A Au,Au) & = ( A\, R_A A u,u) \leq \mu_2( Au,u) \leq \mu^2_2\|u\|_{R_A^{-1}}^2.
\end{align*}
For $ p$ part, it holds
\begin{align*}
(S  + B\, R_AE^{\intercal}B^{\intercal}) R_S (S  + B ER_AB^{\intercal}) 
& \leq 2\kappa_2 \left(S S^{-1} S  + B\, R_AE^{\intercal}B^{\intercal}S^{-1} BER_A B^{\intercal}\right) \\
&\leq 2\kappa_2 \left(S + B\, R_A E^{\intercal}R_A^{-1} E R_A B^{\intercal}  \right) \\
& \leq 2\kappa_2 \left( S + \delta^2 B\, R_A B^{\intercal} \right) \\
& \leq 2\kappa_2 (1 + \delta^2) S \\
& \leq 2\kappa_2^2 (1 + \delta^2) R_S^{-1}.
\end{align*}
where in the second inequality, we have used \eqref{BAB}, and in the third inequality, we have used \eqref{eq:ERA}.
Now, we turn to the second term in \eqref{eq:Uper_bound}. For $ p$ part, we have 
$$
BE \, R_A \, E^{\intercal}B^{\intercal} \leq \delta^2 B\, R_A B^{\intercal} \leq \delta^2 S \leq \kappa_2 \delta^2 R_S^{-1}.
$$
For $u$ part, we have
$$
E^{\intercal}B^{\intercal}R_SBE \leq \kappa_2 E^{\intercal}B^{\intercal}S^{-1} BE \leq \kappa_2 E^{\intercal}R_A^{-1}E \leq \kappa_2\delta^2 R_A^{-1}.
$$
Thus, the inequality \eqref{eq:FoE2} follows with 
$$\Gamma =\left [ 2 \max\left\{ \mu^2_2 + \kappa_2\delta^2, 2\kappa_2^2(1 + \delta^2) + \kappa_2\delta^2\right\}\right ]^{1/2}.$$
\end{proof}

We turn to the estimate of the field-of-values of preconditioned system $\mathcal G \mathcal A$. With an appropriately chosen norm, we can obtain similar results.
\begin{lemma}\label{lem:field_value_GL}
Assume  \eqref{eq:ass_equ} and \eqref{eq:schur_equ} hold, then for any non-zero element $\bs x = (u, p)^{\intercal} \in \mathbb R^n \times \mathbb R^m$, it holds
\begin{equation}
\frac{( \mathcal G\mathcal A\bs x,\bs x)_{\mathcal L\mathcal D \mathcal U}}{\|\bs x\|_{\mathcal L\mathcal D \mathcal U}^2} \geq \gamma,\quad\text{and}\quad \frac{\|\mathcal G\mathcal A\bs x\|_{\mathcal L\mathcal D \mathcal U}}{\|\bs x\|_{\mathcal L\mathcal D \mathcal U}} \leq \Gamma,
\end{equation}
where $\gamma$ and $\Gamma$ are given in Lemma \ref{lem:fiels_pre}.
\end{lemma}
\begin{proof}
For any $\bs x = (u, p)^{\intercal} \in \mathbb R^n \times \mathbb R^m$, it is easy to see that
\begin{align*}
( \mathcal G \mathcal A \bs x,\bs x)_{\mathcal L\mathcal D \mathcal U} & = ( \mathcal L \mathcal D \mathcal H^{-1} \mathcal L^{-1} \mathcal A \bs x,\bs x) = (   \mathcal H^{-1}\mathcal L^{-1} \mathcal A\, \mathcal U^{-1} \mathcal U\bs x,\mathcal U\bs x)_{\mathcal D}.
\end{align*}
Then use inequality \eqref{eq:FoE1}, we have
$$
( \mathcal G\mathcal A\bs x,\bs x )_{\mathcal L \mathcal D\mathcal U} \geq \gamma \|\, \mathcal U\bs x\|_{\mathcal D}^2 = \gamma \| \bs x\|_{\mathcal L \mathcal D\mathcal U}^2.
$$

Similarly we use
\begin{align*}
\|\mathcal G\mathcal A\bs x\|_{\mathcal L\mathcal D\mathcal U} & = \|\mathcal H^{-1} \mathcal L^{-1} \mathcal A\, \mathcal U^{-1} \mathcal U \bs x\|_{\mathcal D}
\end{align*}
and inequality \eqref{eq:FoE2}, to get
$$
\|\mathcal G\mathcal A\bs x\|_{\mathcal L\mathcal D \mathcal U} \leq  \Gamma \| \mathcal U \bs x\|_{\mathcal D}\leq \Gamma \|\bs x\|_{\mathcal L \mathcal D \mathcal U}.
$$
\end{proof}

Using \cite{Elman1982,Saad2003}, \cite[Algorithm 2.2]{Loghin;Wathen2004},  Lemma \ref{lem:fiels_pre} and \ref{lem:field_value_GL}, we conclude the convergence of GMRes with these preconditioners.
\begin{theorem}
Under the assumptions \eqref{eq:ass_equ} and \eqref{eq:schur_equ}, {\rm GMRes} method for the preconditioned system $(\mathcal L\, \mathcal H) ^{-1}\mathcal A\mathcal U^{-1}$ (or $\mathcal G \mathcal A$) is convergent in the $\|\cdot\|_{\mathcal D}$ (or $\|\cdot\|_{\mathcal L\mathcal D \mathcal U}$) norm, respectively. 
\end{theorem}

To obtain a uniform bound, i.e, independent of the size of $A$, for $\delta$ we can use a V-cycle multigrid method as $A$ is formed explicitly. In most scenario, we cannot form the Schur complement $S$ explicitly or it is not worth to form and store $S$ explicitly. Then the challenge is to construct a $R_S$ which is spectrally equivalent to $S$ but easier to compute. Applications to mixed finite element methods for elliptic systems can be found in~\cite{Rusten1996Interior,perugia2000block,Loghin;Wathen2004,Chen;Wu;Zhong:Zhou2016,chen2016fast}. 
 

\end{document}